\documentclass[twocolumn,journal]{IEEEtran}

\usepackage{cite,comment}
\usepackage{amsthm}
\usepackage{mathrsfs}
\usepackage{graphicx}
\usepackage{graphics}
\usepackage{amssymb}
\usepackage{amsmath,mathtools}
\usepackage{color}
\usepackage{xspace}
\usepackage{algpseudocode}
\usepackage{bbm}
\usepackage{comment}
\usepackage{algorithmicx}
\usepackage[caption=false]{subfig}
\usepackage{psfrag}
\usepackage{url}
\usepackage{float}
\usepackage{caption}
\usepackage{soul}
\usepackage{algpseudocode}
\usepackage{algorithm,algorithmicx}
\usepackage{rotating}
\usepackage{chngcntr}
\usepackage{apptools}
\AtAppendix{\counterwithin{thm}{section}}
\usepackage{graphicx}
\usepackage{graphics}
\usepackage{amssymb}
\usepackage{amsmath}
\usepackage{mathtools}
\usepackage{color}
\usepackage{xspace}
\usepackage{algpseudocode}
\usepackage{bbm}
\usepackage{comment}
\usepackage{algorithmicx}
\usepackage{subfig}
\usepackage{psfrag}
\usepackage{tikz}
\usetikzlibrary{shapes,arrows}
\usetikzlibrary{positioning}
\usepackage{rotating}
\usepackage{chngcntr}
\usepackage{apptools}
\usepackage{listings}
\AtAppendix{\counterwithin{thm}{section}}
\usepackage{graphicx}
\usepackage{graphics}
\usepackage{amssymb}
\usepackage{amsmath}
\usepackage{color}
\usepackage{xspace}
\usepackage{algpseudocode}
\usepackage{bbm}
\usepackage{comment}
\usepackage{algorithmicx}
\usepackage{subfig}
\usepackage{psfrag}
\usepackage{tikz}
\usetikzlibrary{shapes,arrows}
\usetikzlibrary{positioning}

\usepackage{enumitem}
\usepackage{graphicx}               % Necessary to use \scalebox
\usepackage{amsmath,amssymb}

\usepackage{graphics}
\usepackage{amssymb}
\usepackage{amsmath}
\usepackage{color}
\usepackage{xspace}
\usepackage{algpseudocode}
\usepackage{bbm}
\usepackage{comment}
\usepackage{algorithmicx}
\usepackage{subfig}
\usepackage{psfrag}
\usepackage{soul}
\usepackage{tikz}
\usetikzlibrary{petri}
\usetikzlibrary{shapes,arrows}
\usetikzlibrary{positioning}

\usepackage{cancel}
\tikzstyle{block}=[draw opacity=0.7,line width=1.4cm]

\newcommand{\oprocendsymbol}{\hbox{$\bullet$}}
\newcommand{\oprocend}{\relax\ifmmode\else\unskip\hfill\fi\oprocendsymbol}

\newcommand{\VV}{\mathcal{V}}
\newcommand{\EE}{\mathcal{E}}
\newcommand{\GG}{\mathcal{G}}

\newcommand{\real}{{\mathbb{R}}}
\newcommand{\reals}{{\mathbb{R}}}
\newcommand{\realpositive}{{\mathbb{R}}_{>0}}
\newcommand{\realnonnegative}{{\mathbb{R}}_{\ge 0}}

\newcommand{\eps}{\epsilon}

\newcommand{\until}[1]{\in\{1,\dots,#1\}}

\newcommand{\vect}[1]{\boldsymbol{\mathbf{#1}}}
\newcommand{\vectsf}[1]{\boldsymbol{\mathbf{\mathsf{#1}}}}

\newcommand{\Diag}[1]{\operatorname{Diag}(#1)}

 \newcommand{\boxend}{\hfill \ensuremath{\Box}}

\newtheorem{thm}{Theorem}[section]

\newtheorem{lem}{Lemma}[section]

\DeclareMathAlphabet{\mathpzc}{OT1}{pzc}{m}{it}

\title{\LARGE \bf Fast model averaging via buffered states and first-order accelerated optimization algorithms}
\author{Amir-Salar Esteki, Hossein Moradian and Solmaz S. Kia, \emph{Senior Member, IEEE} % <-this % stops a space
  \thanks{The authors are with the Department of Mechanical and
    Aerospace Engineering, University of California Irvine, Irvine, CA 92697, USA {\tt\small
      \{aesteki, hmoradia, solmaz\}@uci.edu}. This work was supported by NSF, under CAREER Award ECCS-1653838.}
}

\begin{document}

\maketitle

  \begin{abstract}
  In this letter, we study the problem of accelerating reaching average consensus over connected graphs in a discrete-time communication setting. Literature has shown that consensus algorithms can be accelerated by increasing the graph connectivity or optimizing the weights agents place on the information received from their neighbors. Here, instead of altering the communication graph, we investigate two methods that use buffered states to accelerate reaching average consensus over a given graph. In the first method, we study how convergence rate of the well-known first-order Laplacian average consensus algorithm changes when agreement feedback is generated from buffered states. For this study, we obtain a sufficient condition on the ranges of buffered state that leads to faster convergence. In the second proposed method, we show how the average consensus problem can be cast as a convex optimization problem and solved by first-order accelerated optimization algorithms for strongly-convex cost functions. We construct an accelerated average consensus algorithm using the so-called Triple Momentum optimization algorithm. The first approach requires less global knowledge for choosing the step size, whereas the second one converges faster in our numerical results by using extra information from the graph topology. We demonstrate our results by implementing the proposed algorithms in a Gaussian Mixture Model (GMM) estimation problem used in sensor networks.
  \end{abstract}

\textbf{keywords}: Consensus Algorithm, Accelerated Average Consensus, Delay Systems, Multi-agent systems

\section{Introduction}
In the average consensus problem the objective is to enable a group of communicating agents $\mathcal{V}=\{1,\cdots,N\}$ to arrive at the average of their local input $\mathsf{r}^i\in\real$, i.e., to obtain $\mathsf{r}^{\text{avg}}=\frac{1}{N}\sum_{j=1}^N\mathsf{r}^j$ using local interactions. The solution to this problem is of great importance in various multi-agent applications such as robot coordination~\cite{PY-RAF-KML:08}, sensor fusion~\cite{ROS-JSS:05,ROS:07,WR-UMA:17}, distributed estimation~\cite{SM:07} and  formation control~\cite{JAF-RMM:04}. In these applications, reaching fast to the average consensus is of great interest to reduce the end-to-end delays and also the convergence error caused by premature termination of the algorithm because of time constraints. 

The well-known solution for the average consensus problem is the first-order iterative \emph{Laplacian} algorithm 
\begin{subequations}\label{eq::consensus-orig}
\begin{align}
   & {x}^i(k+1)=x^i(k)-\delta\,\sum\nolimits_{j=1}^N\!\!a_{ij}(x^j(k)-
    x^i(k)),\\
    &~x^i(0)=\mathsf{r}^i,\quad i\in\VV.
\end{align}
\end{subequations}
where $a_{ij}$s are the adjacency weights. In this algorithm, each agent $i$ uses the agreement feedback $\sum\nolimits_{j=1}^N\!a_{ij}(x^j(k)-x^i(k))$ to derive its local agreement state $x^i$ towards $\mathsf{r}^{\text{avg}}$. When the interaction topology of the agents is a connected undirected graph, see Fig.~\ref{fig::graph},~\cite{ROS-JAF-RMM:07} shows that with a proper choice of stepsize $\delta$, executing~\eqref{eq::consensus-orig} guarantees $x^i\!\to\!\mathsf{r}^\text{avg}$, $i\in\VV$, as $k\!\to\!\infty$. Our objective in this paper is to obtain accelerated average consensus algorithms that have a provably faster convergence rate than algorithm~\eqref{eq::consensus-orig}. %when the same stepsize $\delta$ is used. 
We consider two approaches: one using outdated agreement feedback in~\eqref{eq::consensus-orig} and the other by constructing alternative algorithms using the first-order accelerated optimization algorithms for strongly convex unconstrained optimization problems.
 
 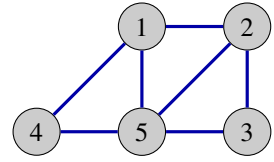
\begin{figure}[h]
\centering
\begin{tikzpicture}[scale=0.7]
\tikzset{% This is the style settings for nodes
    adv/.style={circle,minimum size=.5cm,fill=red!20,draw=red},
    nor/.style={circle,minimum size=.5cm,fill=black!20,draw},
    a1/.style={very thick,red!80!black},
    n1/.style={very thick,blue!65!black}}
\node[nor] (1) at (0,0) {1};
\node[nor] (2) at (2,0) {2};
\node[nor] (3) at (2,-2) {3};
\node[nor] (4) at (-2,-2) {4};
\node[nor] (5) at (0,-2) {5};
\draw[n1][-] (1) -- (2);
\draw[n1][-] (1) -- (5);
\draw[n1][-] (2) -- (3);
\draw[n1][-] (3) -- (5);
\draw[n1][-] (4) -- (5);
\draw[n1][-] (1) -- (4);
\draw[n1][-] (2) -- (5);
\end{tikzpicture}\\
\caption{{\small A connected undirected graph $\mathcal{G}(\VV,\mathcal{E},\mathsf{A})$ with five agents. The adjacency weights of the agents here are $a_{ij}=a_{ji}=1$ if $(i,j)$ is in the edge set $\mathcal{E}$, and $a_{ij}=a_{ji}=0$  otherwise.}}
\label{fig::graph}
\end{figure}

For a multi-agent system with connected  undirected communication graph, the convergence rate of the average consensus algorithm~\eqref{eq::consensus-orig} is tied to the connectivity of the graph~\cite{MF:73} through the spectral radius of matrix $(\vect{I}-\delta \vect{L})$ where $\vect{L}$ is the Laplacian matrix of the graph~\cite{ROS-RMM:04,SSK-BVS-JC-RAF-KML-SM:19}.
Given this connection, 
various studies such as optimal adjacency weight selection for a given topology by maximizing the smallest non-zero eigenvalue of the Laplacian matrix~\cite{LX-SB:04,SB-AG-BP-DS:06} or rewiring the graph to create topologies such as small-world network~\cite{SK-JMF:06,PH-JSB-VG:08} with high connectivity are proposed in the literature. In this letter, instead of altering the communication graph, we investigate two methods that use buffered states to accelerate reaching average consensus over a given graph. 

Our first accelerated consensus algorithm is motivated by evidences in the literature on the positive effect of using buffered feedback on increasing the stability margin and the rate of convergence of the continuous-time linear systems~\cite{HM-SSK:20tac,WQ-RS:13}, which led to use of buffered agreement feedback to accelerate the continuous-time Laplacian average consensus algorithm~\cite{HM-SSK:20tac,YG-MS-CS-NM:16,ZM-YC-WR:10,YC-WR:10, WQ-RS:13}. Since the results obtained for the continuous-time Laplacian algorithm cannot be trivially extended to discrete-time communication setting, we investigate using out-dated feedback in~\eqref{eq::consensus-orig} to increase the convergence rate. More precisely, we explore for what values of non-zero $\mathpzc{d}$ in  
\begin{align}\label{eq::consensus-orig_dated}
   & {x}^i(k+1)=x^i(k)-\delta\,\sum\nolimits_{j=1}^N\!a_{ij}(x^j(k-\mathpzc{d})-
    x^i(k-\mathpzc{d})),\nonumber\\
   &x^i(k)=0~\text{for}~ k\in\{-\mathpzc{d} ,\cdots,-1\},~x^i(0)=\mathsf{r}^i,
\end{align}
$i\in\VV$, we can archive faster convergence than~\eqref{eq::consensus-orig}. Our contribution is to characterize values of $\mathpzc{d}\in\mathbb{Z}_{>0}$ for which convergence is accelerated.
 
Even though our results show that there always exists   $\mathpzc{d}\in\mathbb{Z}_{>0}$ in~\eqref{eq::consensus-orig_dated} that convergence is accelerated, this method has its own limitations because of restricting the structure of the algorithm to the first-order form of algorithm~\eqref{eq::consensus-orig}. This leaves one to wonder whether faster convergence can be achieved by using alternative forms. With such motivation, for example, \cite{TCA-BNO-MJC:08} proposes to improve the rate of convergence by predicting future state values using a weighted summation of current and previous states, denoted as the mixing parameter. However, it requires a complex parameter design procedure since significant improvements in the rate
of convergence are usually achieved by values outside the identified range of the mixing parameter which also requires agents to know extra global information~\cite[Equation (13)]{TCA-BNO-MJC:08}. A simple and more effective approach, however, is reported in~\cite{JB-MF-MM:18}, which casts average consensus problem as a convex optimization problem and uses the  accelerated  Nesterov's optimization method to design a fast-converging average consensus algorithm. Nesterov algorithms~\cite{nesterov2013introductory}, for convex  (denoted here as NAG-C) and for strongly-convex (denoted here as NAG-SC) cost functions, are  gradient-based optimization methods that use the buffered one-step past gradient value to accelerate convergence.  By casting the consensus algorithm as an optimization problem with the cost $\frac{1}{2}\vect{x}^\top\vect{L}\vect{x}$ where $\vect{x}$ is the aggregated agreement state of the agents,~\cite{JB-MF-MM:18} invokes NAG-C method to design its accelerated algorithm. The choice of NAG-C is because $\vect{L}$ of connected graphs is positive semi-definite and thus $\frac{1}{2}\vect{x}^\top\vect{L}\vect{x}$ is a convex function. In this letter, we show that with an alternative modeling approach, we can in fact use the NAG-SC to arrive at a faster converging average consensus algorithm. Our approach also opens the door for use of the so-called Triple Momentum (hereafter denoted as TM) algorithm which is the fastest known globally convergent
gradient-based method for minimizing
strongly convex functions~\cite{BVS-RAF-KML:18}. TM also uses the buffered one-step past gradient value, but has a provably faster convergence than the Nesterov algorithms.

\begin{comment}
\emph{Organization}: 
Notations and preliminaries including a brief review of the relevant properties of the time-buffered discrete-time systems and the graph theoretic definitions are given in Section~\ref{sec::prelim}.  Problem definitions and the objective statements are given in Section~\ref{sec::Prob_formu}, while the main results are given in~Sections~\ref{sec::main} and \ref{sec::acc}. Numerical simulations to illustrate our results are given in Section~\ref{sec::Num_ex}.   Section~\ref{sec::Con} summarizes our concluding remarks.
\end{comment}
%\section{Notations and Preliminaries}\label{sec::prelim}
\noindent \emph{Notations and definitions}:
we let $\reals$, $\realpositive$, $\realnonnegative$, $\mathbb{Z}$, and $\mathbb{C}$
denote the set of real, positive real, non-negative real, integer, and complex numbers, respectively. 
The transpose of a matrix $\vect{A}\in\real^{n\times m}$ is~$\vect{A}^\top$.
The set of eigenvalues of matrix $\vect{A}\in\real^{n\times n}$ is $\lambda(\vect{A})$ and its spectral radius is $\sigma(\vect{A})$. Recall that for a square matrix $\vect{A}$, we have~\cite{RAH-CRJ:90}
\begin{align}\label{eq::spectral_formula}
   \lim_{k\to\infty}\|\vect{A}^k\|^{1/k}=\sigma(\vect{A}),
\end{align}
and when $\vect{A}$ is symmetric, we have $\|\vect{A}\|=\sigma(\vect{A})$.
We follow~\cite{FB-JC-SM:09} to define our graph related terminologies and notations. In addition, we denote $\textup{diam}(\mathcal{G})$ as the diameter of a graph $\mathcal{G}$ which is the length of the shortest path between the most distanced nodes. For an iterative algorithm with states $\|\vect{x}(k)\|$ converging to origin, the \emph{asymptotic convergence factor} is 
%\begin{align}\label{eq::asym_rate}
$\mathfrak{r}=\underset{\vect{x}(0)\neq\vect{0}}{\sup}\lim_{k\to\infty}\left(\frac{\|\vect{x}(k)\|}{\|\vect{x}(0)\|}\right)^{\frac{1}{k}}$
%\end{align}
%\end{thm}
and the associated convergence time is 
$\mathfrak{t}=\frac{1}{\textup{log}(1/\mathfrak{r})}.$
The convergence time $\mathfrak{t}$ represents the (asymptotic) number of steps in which $\|\vect{x}(k)\|$ reduces by the factor $1/\text{e}$. In a network of $N$ agents  with undirected connected graph
topology the graph is denoted by $\mathcal{G}(\mathcal{V},\mathcal{E},\vectsf{A})$ where $\mathcal{V}=\{1,\cdots,N\}$ is the node set, $\mathcal{E}\subset\mathcal{V}\times \mathcal{V}$ is the edge set and $\vectsf{A}=[a_{ij}]$ is the adjacency matrix of the graph. Recall that $a_{ii}=0$, $a_{ij}\in\real_{>0}$ if $j\in\VV$ can send information to agent $i\in\VV$, and zero otherwise. In an undirected graph the connection between the nodes is bidirectional and  $a_{ij}=a_{ji}$ if $(i,j)\in\mathcal{E}$. The maximum degree of a graph is $\mathsf{d}_{\max}=\max\{\sum_{j=1}^Na_{ij}\}_{i=1}^N$. Finally, an undirected graph is connected if there is a path from every agent to every other agent in the network (see e.g.~Fig.~\ref{fig::graph}). The Laplacian matrix of the graph is  $\vect{L}=\text{Diag}(\vectsf{A}\vect{1}_N)-\vectsf{A}$. The Laplacian matrix of a connected undirected graph is a symmetric positive semi-definite matrix that has a simple $\lambda_1=0$ eigenvalue, and the rest of its eigenvalues satisfy $\lambda_1=0< \lambda_2\leq\cdots\leq\lambda_N$.  Moreover, $\vect{L}\vect{1}_N=\vect{0}$.

\section{Problem Definition}\label{sec::Prob_formu}
We study the accelerated average consensus problem over a connected undirected graph $\mathcal{G}(\VV,\EE,\vectsf{A})$. As stated earlier, algorithm~\eqref{eq::consensus-orig} is the well-known solution for the average consensus problem. The admissible stepsize for algorithm~\eqref{eq::consensus-orig} over a connected graph satisfies $\delta\in(0,\frac{2}{\lambda_N})$, for which algorithm~\eqref{eq::consensus-orig} 
converges exponentially fast to the average of the initial conditions of the agents~\cite{ROS-RMM:04}. 
The \emph{asymptotic convergence factor} 
for the Laplacian average consensus algorithm~\eqref{eq::consensus-orig} is $\mathfrak{r}_0=\max\{|1-\delta\lambda_2|,|1-\delta\lambda_N|\}$. For  $\delta\in(0,\frac{1}{\lambda_N}]$, given that $0<\lambda_2\leq\lambda_N$, $\mathsf{r}_0=|1-\delta\lambda_2|$. We can show that the exponential convergence rate of~\eqref{eq::consensus-orig} is equal to $\mathfrak{r}_0+\eps$ for an infinitesimally small $\eps\in\real_{>0}$. Given $\delta\in(0,\frac{1}{\lambda_N})$, if one wants to increase the rate of convergence of algorithm~\eqref{eq::consensus-orig} then the only possible mechanism is to decrease $\delta$, or in another word, increase the frequency of the communicated messages between the agents. The objective in this paper is to investigate algorithms that have provably faster convergence than the Laplacian average consensus algorithm~\eqref{eq::consensus-orig}.
%algorithmbut using the same stepsize $\delta$. 
Our first approach is to investigate using out-dated feedback in~\eqref{eq::consensus-orig}, i.e., using non-zero $\mathpzc{d}$ in~\eqref{eq::consensus-orig_dated}.
Our second approach is to cast the average consensus problem as a convex optimization problem and then seek faster converging algorithms using the first-order accelerated optimization algorithms.

%%%%%%%%%%%%%%%%%%%%%%%%
\section{Accelerated average consensus via outdated agreement feedback}\label{sec::main}
In this section, we study convergence of algorithm~\eqref{eq::consensus-orig_dated} and determine for what values of $\mathpzc{d}\in\mathbb{Z}_{>0}$, this algorithm can converge faster than algorithm~\eqref{eq::consensus-orig}. According to~\cite{HM-SSK:18}, the modified average consensus algorithm~\eqref{eq::consensus-orig_dated} with $\delta\in(0,\frac{2}{\lambda_N})$ is guaranteed to converge when $\mathpzc{d}=1$. 
The results in~\cite{HM-SSK:18} go further to show the admissible range of $\mathpzc{d}\in\mathbb{Z}_{>0}$ for which~\eqref{eq::consensus-orig_dated} converges, see~\cite[Lemma III.4]{HM-SSK:18}. 
\begin{comment}
%this part will be added to the ArXiv version 
The results in~\cite{HM-SSK:18} go further to show that the admissible range of $\mathpzc{d}\in\mathbb{Z}_{>0}$ for which~\eqref{eq::consensus-orig_dated} converges is given as~follows. 
\begin{lem}[Admissible range of $\mathpzc{d}$ for~\eqref{eq::consensus-orig} over connected undirected graphs~\cite{HM-SSK:18}]\label{lem::admissible-d}{\rm
Let $\GG$ be a connected undirected graph. Assume that $\delta\in(0,\frac{2}{\lambda_N})$. Then, for any $\mathpzc{d}\in\{1,\cdots,\bar{\mathpzc{d}}\}$, the average consensus algorithm~\eqref{eq::consensus-orig_dated} satisfies $\lim_{t\to\infty}x^i=\mathsf{x}^\text{avg}(0)$, $i\until{N}$, (the algorithm converges asymptotically) if and only if 
\begin{align}\label{eq::stability con_dis}
\bar{\mathpzc{d}}\!=\!\min\Big\{d\in\mathbb{Z}_{\geq 0}\big|\,d>\hat{d},~~\hat{d}=&\frac{1}{2}\big(\frac{\pi}{2\arcsin(\frac{\delta\,\lambda_i}{2})}-1\big),~\nonumber\\
&\,\,i\in\{2,\cdots,N\}\Big\},
\end{align}
where $\{\lambda_i\}_{i=2}^N$ are the non-zero eigenvalues of $\vect{L}$. }\boxend
\end{lem}
It is shown in~\cite{ISL:05} that the the roots of the characteristic equation~\eqref{eq::char-lambda-i} lie inside the unit circle if and only if $\delta\lambda_i$ lies inside the region of complex plane enclosed by the curve
\begin{align*}
\left\{z\in \mathbb{C}|z=2\vect{i}\sin(\frac{\phi}{2d+1})e^{\vect{i}\phi},-\frac{\pi}{2}\leq\phi\leq\frac{\pi}{2}\right\}.
\end{align*} 

\end{comment}
Next, we determine for what values of $\mathpzc{d}\in\{1,\cdots,\bar{\mathpzc{d}}\}$ the convergence of the modified algorithm~\eqref{eq::consensus-orig_dated} is faster than the convergence of the Laplacian average consensus algorithm. For convenience in our study, we implement the change of variable  $\vect{z}(k)=\vect{T}^\top\vect{x}(k)$ %$\mathsf{x}^\text{avg}(0)=\frac{1}{N}\sum_{i=1}^Nx^i(0)$ 
to write~\eqref{eq::consensus-orig} in the following equivalent form
\begin{subequations}\label{eq::laclacian_equivalent}
\begin{align}
    {z}_1(k+1)&=z_1(k),\quad~ {z}_1(0)=\mathsf{r}^{\text{avg}}=\frac{1}{\sqrt{N}}\sum\nolimits_{j=1}^N\mathsf{r}^j,\label{eq::laclacian_equivalent_z1}\\
z_i(k+1)&\!=z_i(k)\!-\delta\lambda_i z_i(k-\mathpzc{d}),\label{eq::laclacian_equivalent_z2} \\
z_{i}(0)&=[\vect{T}^\top\vect{x}(0)]_i, \quad z_{i}(k)=0~~~ k\in\{-\mathpzc{d},\cdots,-1\}, \nonumber
\end{align}
\end{subequations}
for $i\in\{2,\cdots,N\}$, where \begin{align}\label{eq::T}\vect{T}=\begin{bmatrix}\vect{v}_1&\vect{R}\end{bmatrix},\quad \vect{R}=\begin{bmatrix}\vect{v}_2&\cdots&\vect{v}_N\end{bmatrix}
\end{align}
with $\vect{v}_1=\frac{1}{\sqrt{N}}\vect{1}_N,\vect{v}_2,\cdots,\vect{v}_N$ being the normalized eigenvectors of $\vect{L}$.
Note that $\vect{T}^{-1}=\vect{T}^\top$ and $\vect{T}^\top\vect{L}\vect{T}=\vect{\Lambda}=\Diag{0,\lambda_2,\cdots,\lambda_N}$ because  $\vect{L}$ of a connected undirected graph is a symmetric and real matrix, thus its eigenvectors are mutually orthogonal. Let $\vect{z}_{2:N}=(z_2,\cdots,z_N)^\top$. 
Given~\eqref{eq::laclacian_equivalent_z1}, it follows from   $\vect{z}(k)=\vect{T}^{\top}\,\vect{x}(k)$ that
\begin{align*}\lim_{t\to\infty}\vect{x}(k)=&\,\frac{1}{\sqrt{N}}\lim_{k\to\infty}z_1(k)\vect{1}_N+\lim_{t\to\infty}\vect{R}\,\vect{z}_{2:N}(k)\\=&\,\mathsf{r}^{\text{avg}}\vect{1}_N+\vect{R}\lim_{k\to\infty}\vect{z}_{2:N}(k).
\end{align*}
Therefore, the correctness and the convergence factor of the average consensus algorithm~\eqref{eq::consensus-orig} are determined, respectively, by asymptotic stability and the worst convergence factor of the scalar dynamics in~\eqref{eq::laclacian_equivalent_z2}. 
\begin{comment}
The characteristic equation of the scalar dynamics in~\eqref{eq::laclacian_equivalent_z2} is given  by
\begin{align}\label{eq::char-lambda-i}
    \mathcal{T}(s)=s^{\mathpzc{d}+1}-s^\mathpzc{d}+\delta \lambda_i,\quad i\in\{2,\cdots,N\}.
\end{align}
The roots of~\eqref{eq::char-lambda-i} are all simple except when $\lambda_i=\frac{d^d}{\delta(d+1)^{d+1}}$ \cite{SAK:94}.
It is shown in~\cite{ISL:05} that the the roots of the characteristic equation~\eqref{eq::char-lambda-i} lie inside the unit circle if and only if $\delta\lambda_i$ lies inside the region of complex plane enclosed by the curve
\begin{align*}
\left\{z\in \mathbb{C}|z=2\vect{i}\sin(\frac{\phi}{2d+1})e^{\vect{i}\phi},-\frac{\pi}{2}\leq\phi\leq\frac{\pi}{2}\right\}.
\end{align*} 
Based on this observation and considering the asymptotic stability condition of Theorem~\ref{thm::rate_discrete},~\cite{HM-SSK:18} derives the admissible range of buffer for the average consensus algorithm~\eqref{eq::consensus-orig} as~follows.
\end{comment}
Given~\eqref{eq::laclacian_equivalent}, the convergence factor of the average consensus algorithm~\eqref{eq::consensus-orig_dated} is 
\begin{align}
    \mathfrak{r}_\mathpzc{d}=\max \{\mathfrak{r}_{\mathpzc{d},i}\}_{i=2}^N,
\end{align}
where $\mathfrak{r}_{\mathpzc{d},i}$ is the convergence factor of scalar dynamics~\eqref{eq::laclacian_equivalent_z2}, $i\in\{2,\cdots,N\}$. For each scalar dynamics~\eqref{eq::laclacian_equivalent_z2}, for $i\in\{2,\cdots,N\}$, we have  
$\mathfrak{r}_{\mathpzc{d},i}=\max\{|s^i|\,|\, s^i\in\mathcal{S}\}$, where $\mathcal{S}$ is the set of roots that is determined by 
 the characteristic equation of the scalar dynamics~\eqref{eq::laclacian_equivalent_z2} given  by
\begin{align}\label{eq::char-lambda-i}
    \mathcal{T}(s)=s^{\mathpzc{d}+1}-s^\mathpzc{d}+\delta \lambda_i,\quad i\in\{2,\cdots,N\}.
\end{align}
The roots of~\eqref{eq::char-lambda-i} are all simple except when $\lambda_i=\frac{\mathpzc{d}^\mathpzc{d}}{\delta(\mathpzc{d}+1)^{\mathpzc{d}+1}}$ \cite{SAK:94}. The roots of $\mathcal{T}(s)$ and consequently the size of $\mathfrak{r}_{\mathpzc{d},i}$ depend on $\mathpzc{d}$ and $\delta \lambda_i$. The following result, whose proof is given in the appendix, specifies the values of $\mathpzc{d}>0$ for which the convergence factor of~\eqref{eq::consensus-orig_dated} is less than of~\eqref{eq::consensus-orig}. Recall that smaller convergence factor means faster convergence. 

\begin{thm}[The range of $\mathpzc{d}$ for which the convergence factor of \eqref{eq::consensus-orig_dated} is less than of~\eqref{eq::consensus-orig}]\label{lem::beta-dis_scalar}
Let the communication graph be undirected and connected. For any $\delta \in(0,\frac{2}{\lambda_N})$ the convergence factor of~\eqref{eq::consensus-orig_dated} is less than that of~\eqref{eq::consensus-orig} if
%gConsider  system~\eqref{eq::laclacian_equivalent_z2}, the convergence factor of ~\eqref{eq::consensus-orig_dated} for the solution of \eqref{eq::laclacian_equivalent_z2} is less than buffer-free case, i.e. $d=0$, if  
\begin{align}\label{eq::buffer_bound_dis_rate}
  \mathpzc{d}<\min\left\{\frac{\textup{ln}(\frac{\delta\lambda_i}{\sqrt{(1-\delta\lambda_i)^2+1-2|1-\delta\lambda_i|\cos{\phi}}})}{\textup{ln}(|1-\delta\lambda_i|)},\frac{|1-\delta\lambda_i|}{1-|1-\delta\lambda_i|}\right\}, 
\end{align}
where $\phi\in[0,\frac{\pi}{\mathpzc{d}+1}]$ is the solution of $\frac{\sin{\mathpzc{d}\phi}}{\sin{(\mathpzc{d}+1)\phi}}=\frac{1}{|1-\delta\lambda_i|}$ for all $i\in\{2,\cdots,N\}$. Moreover, the convergence factor is a decreasing function of $\mathpzc{d}\in\mathbb{Z}_{>0}$ if $\frac{\mathpzc{d}^\mathpzc{d}}{(\mathpzc{d}+1)^{\mathpzc{d}+1}}<\{\delta\lambda_i\}_{i=2}^{N}$ holds. \boxend
\end{thm}

%Theorem~\ref{lem::beta-dis_scalar} provides the range of $d$ in terms of the eigenvalues of the Laplacian matrix such that the algorithm~\eqref{eq::consensus-orig_dated} is guaranteed to have lower convergence factor in comparison to the Laplacian  algorithm~\eqref{eq::consensus-orig}. Therefore, given the topology of the network, faster convergence can be achieved by using outdated feedback for a fixed value of stepsize. We also note here that the range of $d$ in~\eqref{eq::buffer_bound_dis_rate} provides the sufficient condition for a faster convergence while the exact range can be beyond that.

%%%%%%%%%%%%%%%%%%%%%%%%
\section{Accelerated average consensus via first-order accelerated optimization algorithms} \label{sec::acc}
In this section, we use the first-order accelerated optimization framework to devise accelerated average consensus algorithms that have proven faster convergence than the well-known average consensus algorithm~\eqref{eq::consensus-orig}. The work in this section is inspired by the results in~\cite{JB-MF-MM:18}.~\cite{JB-MF-MM:18} argued that the conventional average consensus algorithm~\eqref{eq::consensus-orig} can be viewed as the gradient descent algorithm with fixed stepsize $\delta\in(0,\frac{2}{\lambda_N})\subset\real_{>0}$ where the cost function is the agreement potential $f(\vect{x})=\frac{1}{2}\vect{x}^\top\vect{L}\vect{x}$. Note here that $\vect{0}\leq \nabla^2 f(\vect{x})=\vect{L}\leq \lambda_N\vect{I}$. Based on this observation and since the cost function~$f(\vect{x})=\frac{1}{2}\vect{x}^\top\vect{L}\vect{x}$ is convex,~\cite{JB-MF-MM:18} proposes to use the first-order accelerated NAG-C optimization algorithm  
\begin{subequations} \label{eq::ref15}
\begin{align}
    \vect{y}(k+1)&=\vect{x}(k)-\delta\nabla f(\vect{x}(k)),\\
    \vect{x}(k+1)&=\vect{y}(k+1)+\frac{k+1}{k+3}(\vect{y}(k+1)-\vect{y}(k)),
\end{align}
\end{subequations}
with $\vect{x}(0)=\vect{y}(0)=\vectsf{r}$, $\vectsf{r}=[\mathsf{r}^1,\cdots,\mathsf{r}^N]^{\top}$,  where $\nabla f(\vect{x}(k))=\vect{L}\vect{x}(k)$. The choice of coefficient $\frac{k+1}{k+3}$, which tends to one, is fundamental for the argument used by Nesterov to establish the following inverse quadratic convergence rate of $f(\vect{x}(k))-f(\vect{x}^\star)\leq O\left(\frac{1}{\delta \,k^2}\right)$, for any stepsize $0<\delta\leq 1/\lambda_N$, with the best step size being $\delta=\frac{1}{\lambda_N}$. 
%Since the step size is dependent on $\lambda_N$, this algorithm requires all the agents to know this global piece of information.

For a $m-$strongly convex objective $f$ with a $L-$Lipschitz gradient, i.e., $m\vect{I}\leq \nabla^2 f(\vect{x})\leq L\vect{I}$, the NAG-SC algorithm achieves a faster linear convergence of $f(\vect{x}(k))-f(\vect{x}^\star)\leq O((1-\sqrt{\delta m})^k)$ for any stepsize of $\delta\in(0,\frac{1}{L}]$ with the best rate being achieved at $\delta=\frac{1}{L}$. The fastest accelerated globally convergent gradient-based algorithm for a $m-$strongly convex objective $f$ with $L-$Lipschitz gradients however, is the TM algorithm proposed in~\cite{BVS-RAF-KML:18}, which achieves $f(\vect{x}(k))-f(\vect{x}^\star)\leq O\left((1-\sqrt{\frac{m}{L}})^{2k}\right)$. Building on the structure of these two optimization algorithms, we propose the following TM-based accelerated average consensus algorithm
\begin{subequations} \label{eq::TM}
\begin{align}
    \vect{\xi}(k+1)&=(1+\beta)\vect{\xi}(k)-\beta\vect{\xi}(k-1)-\alpha\vect{L}\vect{y}(k),\\
    \vect{y}(k)&=(1+\gamma)\vect{\xi}(k)-\gamma\vect{\xi}(k-1),\\
    \vect{x}(k)&=(1+\delta)\vect{\xi}(k)-\delta\vect{\xi}(k-1),
\end{align}
\end{subequations}
where $(\alpha,\beta,\gamma,\delta)=\Big(\frac{1+\rho}{\lambda_{N}},\frac{\rho^2}{2-\rho},\frac{\rho^2}{(1+\rho)(2-\rho)},\frac{\rho^2}{1-\rho^2}\Big)$, $\rho~=~1-\sqrt{\frac{\lambda_{2}}{\lambda_{N}}}$, $\vect{\xi}(0)=\vect{\xi}(1)=\vectsf{r}$, and the NAG-SC-based  accelerated average consensus algorithm
\begin{subequations} \label{eq::NAG-SC}
\begin{align}
    \vect{x}(k+1)&=\vect{y}(k)-\alpha\vect{L}\vect{y}(k),\\
    \vect{y}(k)&=(1+\beta)\vect{x}(k)-\beta\vect{x}(k-1),
\end{align}
\end{subequations}
with $(\alpha,\beta)=\Big(\frac{1}{\lambda_{N}},\frac{\sqrt{\lambda_{N}}-\sqrt{\lambda_{2}}}{\sqrt{\lambda_{N}}+\sqrt{\lambda_{2}}}\Big)$, $\vect{x}(0)=\vect{x}(1)=\vectsf{r}$. In the following, we prove the convergence of $x^i\to\frac{1}{N}\sum_{i=1}^{N}\mathsf{r}^i$ as $k\to\infty$ for the TM-based algorithm~\eqref{eq::TM}, and omit the proof of the NAG-SC algorithm, since a similar approach can be applied. 

Consider the change of variable $\begin{bmatrix}w_{1} \quad \vect{w}^{\top}_{2:N}\end{bmatrix}^{\top}=\vect{T}^\top\vect{\xi}$, $\begin{bmatrix}q_{1} \quad \vect{q}^{\top}_{2:N}\end{bmatrix}^{\top}=\vect{T}^\top\vect{y}, \begin{bmatrix}p_{1} \quad \vect{p}^{\top}_{2:N}\end{bmatrix}^{\top}=\vect{T}^\top\vect{x},$
where $\vect{T}$ is~\eqref{eq::T}. Then,~\eqref{eq::TM} can be written in an equivalent form
\begin{subequations}\label{eq::TMequ}
\begin{align}
    w_{1}(k+1)&=(1+\beta)w_{1}(k)-\beta w_{1}(k-1), \label{eq::TMequ-a}\\
    q_1(k)&=(1+\gamma)w_1(k)-\gamma w_1(k-1), \label{eq::TMequ-b} \\
    p_1(k)&=(1+\delta)w_1(k)-\delta w_1(k-1), \label{eq::TMequ-c} \\
    \vect{w}_{2:N}(k+1)&=(1+\beta)\vect{w}_{2:N}(k)-\beta \vect{w}_{2:N}(k-1) \nonumber \\
    &-\alpha\vect{L}^{+}\vect{y}_{2:N}(k), \label{eq::TMequ-d}\\
    \vect{p}_{2:N}(k)&=(1+\gamma)\vect{w}_{2:N}(k)-\gamma \vect{w}_{2:N}(k-1), \label{eq::TMequ-e} \\
    \vect{q}_{2:N}(k)&=(1+\delta)\vect{w}_{2:N}(k)-\delta \vect{w}_{2:N}(k-1), \label{eq::TMequ-f},
\end{align}
\end{subequations}
where $\vect{L}^{+}=\vect{R}^{\top}\vect{L}\vect{R}$. The following theorem proves that~\eqref{eq::TM} is a solution for the average consensus problem. The reason that we can use the TM and NAG-SC algorithms to design our accelerated average consensus algorithms reveals itself in the proof of this theorem.  
\begin{thm}\label{thm::TM_based_Convereg}
    Consider a network of $N$ agents communicating over a connected graph. Let the agents of the network implement algorithm~\eqref{eq::TM}. Then, $x^i\to\frac{1}{N}\sum_{i=1}^{N}\mathsf{r}^i$ as $k\to\infty$.
\end{thm}
\begin{proof}
Let us consider the equivalent form of the TM algorithm in~\eqref{eq::TMequ}. From \eqref{eq::TMequ-a}-\eqref{eq::TMequ-c}, it is trivially concluded that $q_1(k)=p_1(k)=w_1(0)$ for $k\in\mathbb{Z}_{>0}$. On the other hand, since $\vect{L}^{+}$ is a positive definite matrix with eigenvalues $\lambda_2,\cdots,\lambda_N$, by virtue of the TM algorithm of~\cite{JB-MF-MM:18}, \eqref{eq::TMequ-d}-\eqref{eq::TMequ-f} minimize the $\lambda_2$-strongly convex function $f(\vect{p}_{2:N})=\frac{1}{2}\vect{p}_{2:N}\vect{L}^{+}\vect{p}_{2:N}$ with $\lambda_N$-Lipschitz gradient to the optimal point $\vect{p}^{\star}_{2:N}=\vect{0}$ with a rate of convergence of $f(\vect{p}_{2:N}(k))-f(\vect{0})\leq O\left(\left(1-\sqrt{\frac{\lambda_2}{\lambda_N}}\right)^{2k}\right)$. Therefore, as $k\to\infty$, $\vect{p}_{2:N}\to\vect{0}$. Considering the change of variables, we know that $q_1(0)=\frac{1}{N}\sum_{i=1}^{N}\mathsf{r}^i$ and thus, $x^i\to\frac{1}{N}\sum_{i=1}^{N}\mathsf{r}^i$ as $k\to\infty$.
\end{proof}

We showed in the result above that the algorithm presented in~\eqref{eq::TM} solves the average consensus problem. A Similar result can be established for~\eqref{eq::NAG-SC}. Based on the developments in~\cite{JB-MF-MM:18}, it is proved that \eqref{eq::ref15} converges to the average of local reference values asymptotically and faster than the popular solution~\eqref{eq::consensus-orig}. Moreover, from~\cite{BVS-RAF-KML:18}, we know that the TM algorithm benefits the fastest exponential convergence rate among other accelerated first-order gradient methods in optimization, e.g.,~\eqref{eq::ref15}. We illustrate this comparison in a numerical example in the following section with also simulating the algorithm in~\cite{TCA-BNO-MJC:08}.

Similar to~\cite{TCA-BNO-MJC:08}, we note that faster convergence in~\eqref{eq::TM} and \eqref{eq::NAG-SC} comes with requiring the agents to know $\lambda_2$ and $\lambda_N$ globally in order to compute $(\alpha,\beta,\gamma,\delta)$. The knowledge of $\lambda_N$ to choose the stepsize $\delta$ is universal among all discrete-time average consensus algorithms, including the Laplacian algorithm~\eqref{eq::consensus-orig}. Such knowledge is especially useful in choosing the best stepsize for the fastest convergence. In practice, instead of $\lambda_N$, often its upper-bound $\bar{\lambda}_N=2\mathsf{d}_{\max}$, which is easier to compute, is  used~\cite{ROS-JAF-RMM:07}. On the other hand, $\lambda_2$  can be either computed through a dedicated distributed algorithm, see, e.g., \cite{yang2010decentralized}, or replaced by a lower bound such as $\underbar{$\lambda$}_2=\frac{4}{N \textup{diam}(\mathcal{G})}$, see e.g., \cite{mohar1991eigenvalues}.

%Therefore, by having extra global information, faster convergence can be achieved using the presented accelerated solutions.

\section{Performance demonstration in distributed Gaussian Mixture Model (GMM) Estimation %Distributed target density distribution estimation 
}\label{sec::Num_ex}
Average consensus is an important tool to enable many distributed schemes in networked systems. To demonstrate the benefit of using our accelerated average consensus, we conduct a simulation study of a distributed expectation-maximization (EM) algorithm used in sensor networks to obtain a Gaussian Mixture Model (GMM) of a set of observed targets. In our case study, the setting consists of $N$ agents observing the location $\vectsf{p}\in\real^2$ of $M$ targets in a 2D plane. The agents want to collaboratively obtain the GMM model of the distribution of the targets, i.e., obtain the weight, mean and covariance of basis of the GMM model, i.e., $(\pi_l,\vect{\mu}_l,\Sigma_l)$ in $\hat{f}(\vectsf{p})=\sum_{l=1}^{N_s}\pi_l\mathcal{N}(\vectsf{p}|\vect{\mu}_{l},\vect{\Sigma}_l)$, where $N_s$ is the number of the bases of the GMM model which is known to all agents. A popular method to construct a GMM with a determined number of bases $N_s$ from observed data is the EM algorithm~\cite{dempster1977maximum} which is an iterative method that alternates between an expectation (E) step and a maximization (M) step. In the E-step, the posterior probability is computed as
\begin{align}\label{eq::posterior}
    \zeta_{ln}:=\textup{Pr}(z=l|\vectsf{p}_n)=\frac{\pi_{l}\mathcal{N}(\vectsf{p}_n|\vect{\mu}_{l},\vect{\Sigma}_l)}{\sum_{j=1}^{N_s}\pi_{j}\mathcal{N}(\vectsf{p}_n|\vect{\mu}_{j},\vect{\Sigma}_j)}
\end{align}
using the target points $\vectsf{p}_n$, $n\in\{1,\cdots,M\}$ and the current values of $\{\pi_{l},\vect{\mu}_{l},\vect{\Sigma}_l\}_{l=1}^{N_s}$. Next, in the M-step, the parameters of each $l\in\{1,\cdots,N_s\}$ are updated by
\begin{subequations}\label{eq::update}
\begin{align}
    \pi_l&=\frac{\sum_{n=1}^{M}\zeta_{ln}}{M}, \label{eq::updateA} \\
    \vect{\mu}_{l}&=\frac{\sum_{n=1}^{M}\zeta_{ln}\vectsf{p}_n}{\sum_{n=1}^{M}\zeta_{ln}}, \\
    \vect{\Sigma}_l&=\frac{\sum_{n=1}^{M}\zeta_{ln}(\vectsf{p}_n-\vect{\mu}_l)(\vectsf{p}_n-\vect{\mu}_l)^{\top}}{\sum_{n=1}^{M}\zeta_{ln}},
\end{align}
\end{subequations}
using the current values of $\zeta_{ln}$. In the distributed EM algorithm, each agent only observes a $M^i$ number of targets; the set is given by $\mathcal{M}^i\subset\{1,\cdots,M\}$ where $\cup_{j=1}^N\mathcal{M}^j=\{1,\cdots,M\}$, $\mathcal{M}^i\cap\mathcal{M}^j=\{\}$. The distributed EM algorithms assume that each agent has a local copy of the GMM parameters~\cite{valdeira2022decentralized, altilio2019distributed}. Each agent $i\in\{1,\cdots,N\}$ performs the E-step in~\eqref{eq::posterior} locally using its own GMM parameters for $n\in\mathcal{M}^i$. However, the summation terms in~\eqref{eq::update} are fragmented among the agents. Therefore, agents use a set of three average consensus algorithms to compute the summation terms that appear in the M-step~\eqref{eq::update}.

In our numerical example, the number of agents is $N=20$, and the agents communicate over a ring graph. These agents observe $M=1000$ target points in a rectangle area of $(-80, 80)\times (-60, 60)$. The target points are drawn from a GMM model with $N_s=12$ so that we can check the accuracy of the estimated GMM via distributed EM algorithms against this true model. Each agent initializes its $\{\pi_{l},\vect{\mu}_{l},\vect{\Sigma}_l\}_{l=1}^{N_s}$ locally. Let us denote $T_{\textup{EM}}~=~10$ as the number of iterations in the EM algorithm, and $T_{\textup{consensus}}$ as the number of the consensus steps performed in each iteration of the M-step.
First, we compare the performance of the EM algorithm when it uses the Laplacian algorithm in~\eqref{eq::consensus-orig} (Laplacian-based EM) vs. when it uses our proposed TM-based algorithm~\eqref{eq::TM}. We conduct a set of four simulations for  $T_{\textup{consensus}}\in\{8,15,30,50\}$. When using the TM algorithm, we consider two cases. In one, we assume that the agents know $\lambda_2$ and $\lambda_N$ to compute the parameters of the TM algorithm (referred to as TM-based EM (exact)), and in the other case, we assume that agents replace $\lambda_2=\frac{4}{N \textup{diam}(\mathcal{G})}$ by its lower bound and $\lambda_N$ with its upper bound $2\mathsf{d}_{\max}$ (referred as TM-based EM (via bounds)). 
%We implement the Laplacian algorithm in~\eqref{eq::consensus-orig} and our proposed TM-based algorithm~\eqref{eq::TM} to estimate the values of~\eqref{eq::update}. The TM-based algorithm is simulated using two different stepsizes $\rho=1-\sqrt{\frac{\lambda_2}{\lambda_N}}$ and $\rho=1-\sqrt{\frac{\underbar{\lambda}_2}{\bar{\lambda}_N}}$ which are denoted by the unbounded TM and bounded TM, respectively. 
%The bounded TM is useful for scenarios where deriving $\lambda_2$ and $\lambda_N$ distributively is computationally. 
Due to the limited space, we only show the results generated by agent 1 in all the simulations; the other agents have similar results. Fig.~\ref{fig::GMM} depicts the 3$\sigma$-plot of the bases of the estimated GMM for $T_{\textup{consensus}}=8$. Here, the thin gray, the thin colored, and the thick colored ellipses represent, respectively, the true GMM model, the estimated GMM model using the Laplacian-based EM, and the estimated GMM model using the TM-based EM (via bounds). As seen in Fig.~\ref{fig::GMM}, the results generated by the TM-based EM (via bounds) are closer to the true model, especially in some bases, such as the top right purple and the bottom center magenta one. The results for TM-based EM (exact) are not shown to reduce clutter in Fig.~\ref{fig::GMM}. To better show the accuracy of each estimated GMM model, Fig.~\ref{fig::Loglikelihood} depicts the maximum Log-likelihood of the distributed EM algorithms in comparison to the central EM. By using the same number of communications, the TM-based EM algorithms, even when we use the bounds instead of the exact values for $\lambda_2$ and $\lambda_N$ achieve better results compared to the Laplacian algorithm in the sense that the maximum log-likelihood of the TM-based estimates are closer to the central solution. This difference is especially larger in cases where the number of communications is limited, e.g., $T_{\textup{consensus}}=8$.
\begin{figure}[h]
    \centering
    \subfloat[]{\includegraphics[trim= 1pt 5pt 0 0 ,clip,width=1\linewidth]{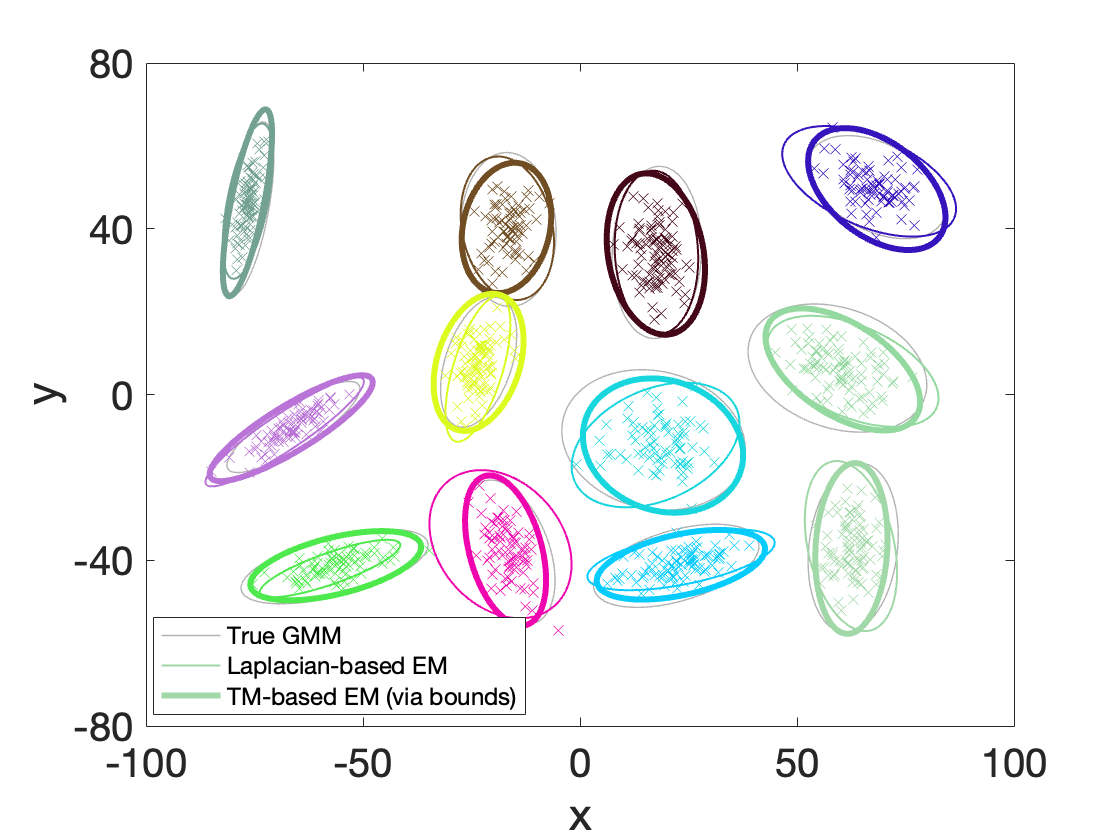}\label{fig::GMM}}\\
    \subfloat[]{\includegraphics[trim= 1pt 5pt 0 0 ,clip,width=1\linewidth]{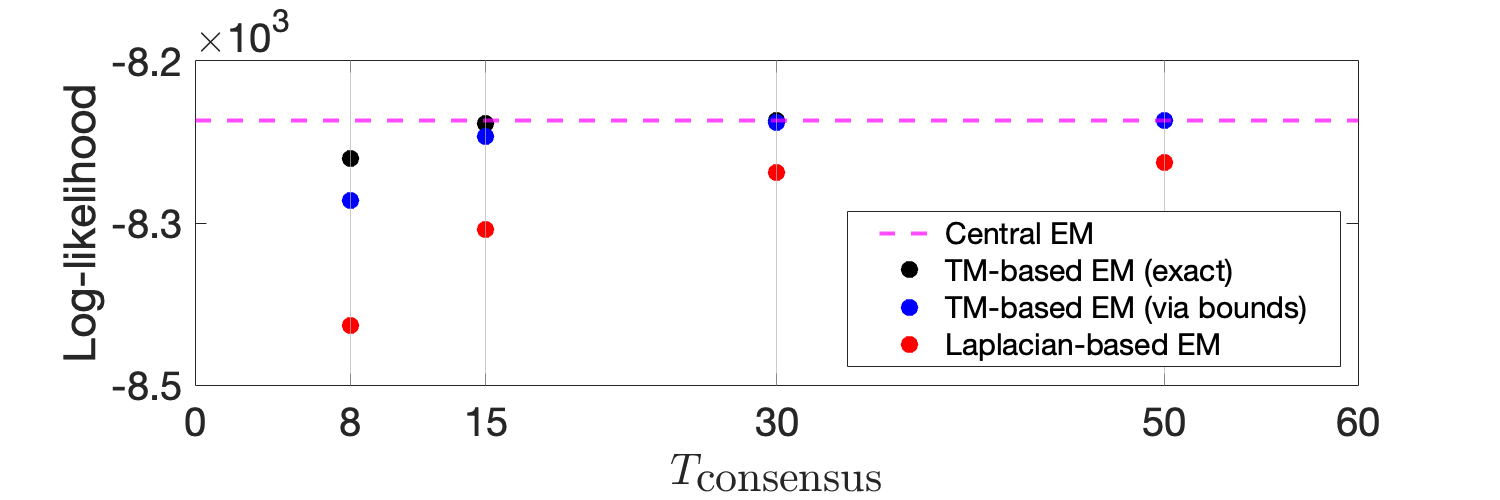}\label{fig::Loglikelihood}}
    \caption{{\small Plot (a) depicts the GMM estimates of Agent 1 after $T_{\textup{EM}}=10$ iterations of the EM algorithm. %The thin gray, thin colored and thick colored ellipses indicate the 3$\sigma$-plot of the bases of the true GMM model, the estimated GMM model using the Laplacian algorithm and the estimated GMM model using the bounded TM-based algorithm. 
    Plot (b) represents the maximum Log-likelihood of the estimated models of the central EM, TM-based EM (exact), TM-based EM (via bounds), and Laplacian-based EM.
    }}
\end{figure}

In order to compare the convergence rates of the proposed algorithms against the existing ones in the literature, Fig.~\ref{fig::convergence} illustrates the evolution of the estimates of $\pi_l$, for $l=1$, over $T_{\textup{consensus}}=50$ steps in the first iteration of the EM algorithm. Here, the convergence error trajectories of the consensus algorithms is denoted by $e(k)=\textup{log}\sum_{i=1}^{N}(\pi_1^i(k)-\pi_1^{\star})^2$, where $\pi_1^{\star}$ is the central solution obtained by~\eqref{eq::updateA}. It is shown that the TM algorithm achieves the fastest convergence rate and also, the buffered Laplacian algorithm with $\mathpzc{d}=2$ is faster than the original Laplacian algorithm without using buffered feedback.
\begin{figure}[h]
    \centering
    {\includegraphics[trim= 1pt 5pt 0 0 ,clip,width=.8\linewidth]{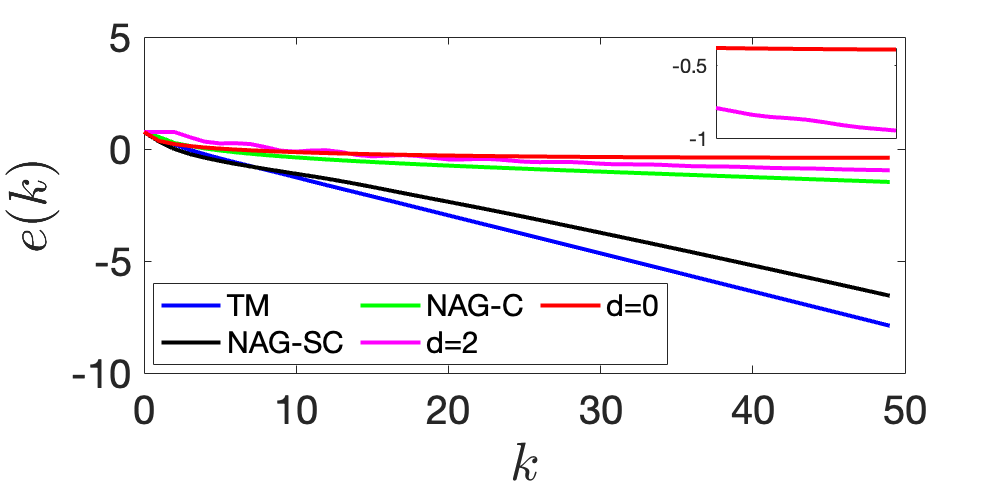}}
    \caption{{This plot shows the convergence error of $e(k)$ for obtaining $\pi_1^i$ with respect to the actual average computed as~\eqref{eq::updateA}. Five different average consensus algorithms have been implemented and compared as above. For better visual comparison, the convergence error trajectories of the Laplacian algorithm with $\mathpzc{d}=0$ and $\mathpzc{d}=2$ for $k=\{40,\cdots,50\}$, is maximized in the top right corner.
    }}
\end{figure}\label{fig::convergence}
\section{Accelerated distributed linear regression}
In this section, we investigate the numerical examples that show the effect of outdated feedback data and the implementation of accelerated first-order optimization methods. We use real data from~\cite{utts2021mind} to solve a linear regression problem by reformulating it as two average consensus problems. In the first example, the common Laplacian average consensus algorithm is compared with the proposed algorithm in~\eqref{eq::consensus-orig_dated} with different number of buffers to analyze the effect of $d$ on convergence. In the second example, convergence of the proposed accelerated algorithms~\eqref{eq::ref15}-\eqref{eq::NAG-SC} are compared against \eqref{eq::consensus-orig_dated} and the algorithm in~\cite{TCA-BNO-MJC:08}.

The dataset in~\cite{utts2021mind} has a size of 50 points for the 50 states in the United States. The variables are $y$ which is year 2002 birth rate per 1000 females 15 to 17 years old, and $x$ which is the poverty rate (the percent of the state’s population living in households with incomes below the federally defined poverty level). The objective is to find a linear relation between $x$ and $y$, i.e., solve the following problem:
\begin{align}\label{eq::regression}
    \textup{min}_{a,b\in\mathbb{R}}\sum\nolimits_{i=1}^{50}\|y^i-(ax^i+b)\|^2
\end{align}
for $a$ and $b$. Since our goal is to study the effect of buffer in convergence, we simplify the problem and assume to know the optimal value of $b$. Thus, we only solve for the variable $a$ as the slope of the fitted line. Here, $b=4.267$. By taking a derivative from~\eqref{eq::regression} with respect to $a$, setting it to zero and substituting the value of $b$, we conclude that
\begin{align}\label{eq::RegressionSolution}
    a=\frac{\sum_{i=1}^{50}x^i(y^i-b)}{\sum_{i=1}^{50}(x^i)^2}.
\end{align}
Let us now solve this problem while the dataset is distributed over a network of 5 agents where each agent has access to 10 arbitrary state data points. Agents of the network communicate over a connected graph depicted in Fig.~\ref{fig::graph}. This solution can be achieved by solving two average consensus problems for the nominator and the denominator, and computing the division. Let $\eta_1^i(k)$ and $\eta_2^i(k)$ denote agent $i$'s local estimate of the nominator and the denominator of~\eqref{eq::RegressionSolution}, respectively. The reference values of the first average consensus problem are $\mathsf{r}_1^1=\sum_{i=1}^{10}x^i(y^i-b)$, $\mathsf{r}_1^2=\sum_{i=11}^{20}x^i(y^i-b)$,$\cdots$, $\mathsf{r}_1^5=\sum_{i=41}^{50}x^i(y^i-b)$. The same process is applied to the second average consensus problem with the reference values $\mathsf{r}_2^1=\sum_{i=1}^{10}(x^i)^2$, $\mathsf{r}_2^2=\sum_{i=11}^{20}(x^i)^2$,$\cdots$, $\mathsf{r}_2^5=\sum_{i=41}^{50}(x^i)^2$. Each agent computes the estimate $\hat{a}^i(k)=\frac{\eta_1^i(k)}{\eta_2^i(k)}$ at every step $k\in\mathbb{Z}_{\geq0}$.

\emph{Outdated feedback}:
Let the agents implement algorithm~\eqref{eq::consensus-orig_dated} with three different values of $d=\{1,5,10\}$ and a buffer-free case. Fig.~\ref{fig::error1} depicts the convergence error of agent's trajectories with respect to the optimal value in~\eqref{eq::RegressionSolution}, i.e., $e(k)=\textup{log}\sum_{i=1}^{N}(\hat{a}^i(k)-a)^2$. The Green line, representing the implementation of one step buffer, as seen in the figure, converge faster than the buffer-free case. By increasing $d$ and using further outdated feedback, faster convergence is achieved. This shows the effect of buffer in our analysis. However, as mentioned previously, exceeding $\bar{d}$ may result in divergence or slower convergence. The Turquoise trajectories with $d=10$ illustrate the fluctuation in convergence.

\emph{Accelerated consensus via first-order accelerated optimization algorithms}:
Next, we compare the convergence rate of the accelerated algorithms~\eqref{eq::ref15}, \eqref{eq::TM} and \eqref{eq::NAG-SC} with that of the accelerated algorithm proposed in~\cite{TCA-BNO-MJC:08} and also algorithm~\eqref{eq::consensus-orig_dated}. Let the agents of the network solve two average consensus problems to reach $a$ globally. Fig.~\ref{fig::error2} shows the convergence error trajectories $e(k)=\textup{log}\sum_{i=1}^{N}(\hat{a}^i(k)-a)^2$ reaching the agreement similar to the previous example. Algorithm~\eqref{eq::consensus-orig_dated} with $d=5$ converges slower compared to others, while the TM algorithm converges the fastest. Despite using   the optimal parameters for the algorithm of  in~\cite{TCA-BNO-MJC:08}, still TM-based algorithm delivers the fastest convergence.
\begin{figure}[h]
    \centering
    \subfloat[]{\includegraphics[trim= 1pt 5pt 0 0 ,clip,width=.9\linewidth]{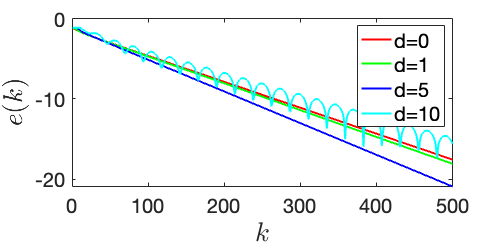}\label{fig::error1}} \vspace{-0.1in}\\
    \subfloat[]{\includegraphics[trim= 1pt 5pt 0 0 ,clip,width=.9\linewidth]{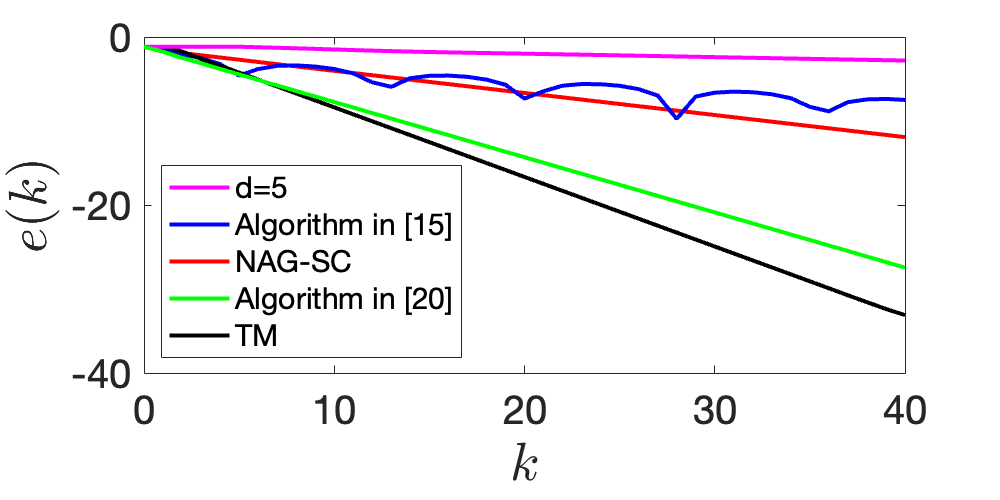}\label{fig::error2}}
    \caption{{\small Plot (a) compares the effect of different values of buffer ($d=\{0,1,5,10\}$) on the convergence error $e(k)$ over $k$. We can see that $d=5$ generates the best results. Plot (b) shows the comparison between the optimization inspired algorithms~\eqref{eq::ref15}-\eqref{eq::NAG-SC} and the algorithm~\eqref{eq::consensus-orig_dated} with $d=5$. As expected from the results in section~\ref{sec::acc}, algorithm~\eqref{eq::TM} converges faster than the other methods.
    }}\vspace{-0.1in}
    \label{fig::error}
\end{figure}
\section{Conclusion}\label{sec::Con}\vspace{-2pt}
In this letter, we proposed two methods to accelerate reaching the solution in the average consensus problem over connected graphs in a discrete-time communication setting. In contrast to some existing methods in the literature where graph connectivity is increased or edge weights are optimized for faster convergence, we used buffered states to accelerate reaching average consensus. First, we proposed to use buffered states in the well-known Laplacian algorithm in order to achieve a faster convergence rate. Furthermore, we obtained the admissible ranges of delay that allow agents to reach the solution. In the second method, we proposed two accelerated average consensus algorithms inspired by the NAG-SC and TM algorithms. We showed that the average consensus algorithm can be cast as a convex optimization problem which can be solved using the NAG-SC and TM algorithms. To demonstrate the efficacy of our results, we
conducted a simulation study of a distributed EM algorithm which is used vastly, e.g., in sensor networks, to estimate a Gaussian Mixture Model of a set of observed targets. By measuring the maximum Log-likelihood of the estimates of the GMM model using the proposed algorithms against other algorithms in the literature, we showed that the TM-based EM achieves more accurate estimates. The estimated models of the distributed algorithms were compared with respect to the central solution of the EM algorithm.

\section*{Appendix}
We use the following auxiliary lemma in the proof of Theorem~\ref{lem::beta-dis_scalar}, which we present afterwards. 

\begin{lem}[Location of roots of $s^{\mathpzc{d}+1}-s^\mathpzc{d}+c=0$~\cite{SAK:94}]\label{thm::dis_time}{\rm
Let $\mathcal{S}=\{s^i\in\mathbb{C}|s^{\mathpzc{d}+1}-s^\mathpzc{d}+c=0\}$ and $s^1=\max\{|s^i|\,|s^i\in\mathcal{S}\}$ for any $i\in\{1,\cdots,\mathpzc{d}+1\}$. Then for  $\mathpzc{d}\in\mathbb{Z}_{>1}$   and $c\in\real_{>0}$ all the roots  are inside the disk $|s|<\frac{1}{|a|}$ if and only if $|a|<\frac{\mathpzc{d}+1}{\mathpzc{d}}$ and
\begin{align}
    \frac{(|a|-1)}{|a|^{\mathpzc{d}+1}}<c<\frac{\sqrt{a^2+1-2|a|\cos\phi}}{|a|^{\mathpzc{d}+1}}
\end{align}
where $\phi\in[0,\frac{\pi}{\mathpzc{d}+1}]$ is the solution of $\frac{\sin(\mathpzc{d}\phi)}{\sin((\mathpzc{d}+1)\phi)}=\frac{1}{|a|}$.  

Moreover if $0<c<\frac{\mathpzc{d}^\mathpzc{d}}{(\mathpzc{d}+1)^{\mathpzc{d}+1}}$, as $c$ increases then the  value of 
$|s^1|$ decreases while the absolute value of 
all the other roots increase. In addition the smallest value of $|s^1|$   occurs at  $c=\frac{\mathpzc{d}^\mathpzc{d}}{(\mathpzc{d}+1)^{\mathpzc{d}+1}}$ where $|s^1| =\frac{\mathpzc{d}}{\mathpzc{d}+1}$.}\boxend
\end{lem}

\noindent Table.~\ref{table.1} shows different values of $\frac{\mathpzc{d}^\mathpzc{d}}{(\mathpzc{d}+1)^{\mathpzc{d}+1}}$ for a given $\mathpzc{d}\in\{1,\cdots,5\}$.

\begin{proof}[Proof of Theorem~\ref{lem::beta-dis_scalar}]
Notice that for $\mathpzc{d}=0$ the asymptotic convergence factor of \eqref{eq::laclacian_equivalent_z2} is equal to $\mathsf{r}_0=|1-\delta\lambda_i|$. Hence, our aim is to find the values of $\mathpzc{d}$ such that $\mathsf{r}_\mathpzc{d}<\mathsf{r}_{0}$, which means that the roots of the characteristic equation~\eqref{eq::char-lambda-i} lie inside the disk, $|s|<|1-\delta\lambda_i|$.
Theorem~\ref{thm::dis_time} implies that this holds if and only if 
  \begin{subequations}
  \begin{align}
  &\frac{1}{|1-\delta\lambda_i|}<\frac{\mathpzc{d}+1}{\mathpzc{d}}\label{eq::dis_scal_proof_a}\\
  &\frac{|\frac{1}{1-\delta\lambda_i}|\!-\!1}{|\frac{1}{1-\delta\lambda_i}|^{\mathpzc{d}+1}}<\!\delta\lambda_i\!<\!\frac{\sqrt{\frac{1}{(1-\delta\lambda_i)^2}\!+\!1\!-\!2|\frac{1}{1-\delta\lambda_i}|\cos{\phi}}}{|\frac{1}{1-\delta\lambda_i}|^{\mathpzc{d}+1}},\label{eq::dis_scal_proof_b}
 % \end{cases}
 \end{align}
  \end{subequations}
  where $\phi\in[0,\frac{\pi}{\mathpzc{d}+1}]$ is the solution of $\frac{\sin{\mathpzc{d}\phi}}{\sin{(\mathpzc{d}+1)\phi}}=\frac{1}{|1-\delta\lambda_i|}$.~For any $\delta\in(0,\frac{2}{\lambda_N})$, 
  %By equation \eqref{eq::r0} 
  we know $-1<1-\delta\lambda_i<1$. From~\eqref{eq::dis_scal_proof_a} we get $\mathpzc{d}<\frac{|1-\delta\lambda_i|}{1-|1-\delta\lambda_i|}$. The left-side inequality of~\eqref{eq::dis_scal_proof_b}~is satisfied for any $\delta\in(0,\frac{2}{\lambda_N})$ with $|1-\delta\lambda_i|\!<\!1$. By some algebraic manipulation, the right-side inequality deduces to $$\mathpzc{d}\!<
  \!\frac{\text{ln}(\frac{\delta\lambda_i}{\sqrt{(1-\delta\lambda_i)^2+1-2|1-\delta\lambda_i|\cos{\phi}}})}{\text{ln}(|1-\delta\lambda_i|)},$$  which concludes \eqref{eq::buffer_bound_dis_rate}. The last~statement is the direct application of Lemma~\ref{thm::dis_time} for~$c=\delta\lambda_i$.
\end{proof}
\begin{table}[]
\caption{The  values of $\frac{\mathpzc{d}^\mathpzc{d}}{(\mathpzc{d}+1)^{\mathpzc{d}+1}}$ for a given $\mathpzc{d}$.%\solmaz{make the table horizontal}
    }
    \label{table.1}
    \centering
   \begin{tabular}{ |c|c|c|c|c|c| } 
 \hline
 $\mathpzc{d}$ & $1$ & $2$ & $3$ & $4$ & $5$\\ 
 \hline
 $\frac{\mathpzc{d}^\mathpzc{d}}{(\mathpzc{d}+1)^{\mathpzc{d}+1}}$ 
  & $0.250$  
  & $0.148$   
  & $0.105$   
  & $0.082$   
  & $0.067$  \\ 
 \hline
 \end{tabular}
\end{table}

\bibliographystyle{ieeetr}

\end{document}